\documentclass[12pt,reqno]{amsart}

\usepackage{amssymb}

\usepackage[all]{xy}

\numberwithin{equation}{section}

\newtheorem{theorem}{Theorem}[section]
\newtheorem{proposition}[theorem]{Proposition}
\newtheorem{lemma}[theorem]{Lemma}
\newtheorem{corollary}[theorem]{Corollary}

\theoremstyle{definition}
\newtheorem{definition}[theorem]{Definition}

\DeclareMathOperator*{\Ind}{Ind}

\begin{document}

\baselineskip=15pt

\title[Ramified covering maps of singular curves]{Ramified covering maps of singular
curves and stability of pulled back bundles}

\author[I. Biswas]{Indranil Biswas}

\address{Department of Mathematics, Shiv Nadar University, NH91, Tehsil
Dadri, Greater Noida, Uttar Pradesh 201314, India}

\email{indranil.biswas@snu.edu.in, indranil29@gmail.com}

\author[M. Kumar]{Manish Kumar}

\address{Statistics and Mathematics Unit, Indian Statistical Institute,
Bangalore 560059, India}

\email{manish@isibang.ac.in}

\author[A. J. Parameswaran]{A. J. Parameswaran}

\address{School of Mathematics, Tata Institute of Fundamental
Research, Homi Bhabha Road, Mumbai 400005, India}

\email{param@math.tifr.res.in}



\date{}

\begin{abstract}
Let $f\, :\, X\, \longrightarrow\, Y$ be a generically smooth nonconstant morphism between irreducible projective
curves, defined over an algebraically closed field, which is \'etale on an open subset of $Y$ that contains both
the singular locus of $Y$ and the image, in $Y$, of the singular locus of $X$. We prove that the following
statements are equivalent:
\begin{enumerate}
\item The homomorphism of \'etale fundamental groups $$f_*\, :\, \pi_1^{\rm et}(X)
\,\longrightarrow\,\pi_1^{\rm et}(Y)$$ induced by $f$ is surjective.

\item There is no nontrivial \'etale covering $\phi\, :\, Y'\, \longrightarrow\, Y$ admitting
a morphism $q\,:\, X\, \longrightarrow\, Y'$ such that $\phi\circ q \,=\, f$.

\item The fiber product $X\times_Y X$ is connected.

\item $\dim H^0(X,\, f^*f_* {\mathcal O}_X)\,=\, 1$.

\item ${\mathcal O}_Y\, \subset\, f_*{\mathcal O}_X$ is the maximal semistable subsheaf.

\item The pullback $f^*E$ of every stable sheaf $E$ on $Y$ is also stable.
\end{enumerate}
\end{abstract}

\maketitle

\section{Introduction}

Let $k$ be an algebraically closed field. Let $X$ and $Y$ be irreducible smooth projective curves and
$f\, :\, X\, \longrightarrow\, Y$ a generically smooth nonconstant map. In \cite{BP} it was proved
that the following six statements are equivalent:
\begin{enumerate}
\item The homomorphism between \'etale fundamental groups
$$f_*\, :\, \pi_1^{\rm et}(X)\,\longrightarrow\, \pi_1^{\rm et}(Y)$$
induced by $f$ is surjective.

\item The map $f$ does not factor through some nontrivial \'etale cover of $Y$ (in particular, $f$ is not
nontrivial \'etale).

\item The fiber product $X\times_Y X$ is connected.

\item $\dim H^0(X,\, f^*f_*{\mathcal O}_X)\,=\,1$.

\item The maximal semistable subbundle of the direct image $f_*{\mathcal O}_X$ is ${\mathcal O}_Y$.

\item For every stable vector bundle $E$ on $Y$, the pullback $f^*E$ is also stable.
\end{enumerate}

Our aim here is to extend this to the context of generically smooth morphisms between singular curves.
Examples show that some conditions are needed in order to be able extend the above result to the
context of generically smooth morphisms between singular curves; see Section \ref{se-ex}. To address this, we
consider maps that are \'etale over singular locus (EOSL for short).

Let $X$ and $Y$ be reduced irreducible projective curves over $k$, and let
$$
f\, :\, X\, \longrightarrow\, Y
$$
be a generically smooth nonconstant morphism. The singular loci of $X$ and $Y$ are denoted by $S_X$ and
$S_Y$ respectively. The map $f$ is called EOSL if $f$ is \'etale over a neighborhood of $S_Y\cup f(S_X)$.

Let $\widehat{X}$ and $\widehat{Y}$ be the normalizations of $X$ and $Y$ respectively. A map
$f\, :\, X\, \longrightarrow\, Y$ produces a map $f'\, :\, \widehat{X}\, \longrightarrow\, \widehat{Y}$.

Let $f\, :\, X\, \longrightarrow\, Y$ be an EOSL map. We prove that the following seven statements are
equivalent (see Theorem \ref{thm1} and Theorem \ref{thm2}):
\begin{enumerate}
\item $f$ is genuinely ramified.

\item The map $f'\, :\, \widehat{X}\, \longrightarrow\, \widehat{Y}$ is genuinely ramified.

\item The homomorphism of \'etale fundamental groups $$f_*\, :\, \pi_1^{\rm et}(X)
\,\longrightarrow\,\pi_1^{\rm et}(Y)$$ induced by $f$ is surjective.

\item There is no nontrivial \'etale covering $\phi\, :\, Y'\, \longrightarrow\, Y$ admitting
a morphism $q\,:\, X\, \longrightarrow\, Y'$ such that $\phi\circ q \,=\, f$.

\item The fiber product $X\times_Y X$ is connected.

\item $\dim H^0(X,\, f^*f_* {\mathcal O}_X)\,=\, 1$.

\item The pullback $f^*E$ of every stable sheaf $E$ on $Y$ is also stable.
\end{enumerate}

\section{EOSL maps and semistability}

The base field $k$ is assumed to be algebraically closed. There is no assumption on its characteristic.

Let $X$ and $Y$ be reduced irreducible projective curves over $k$, and let
$$
f\, :\, X\, \longrightarrow\, Y
$$
be a generically smooth nonconstant morphism. Let $B_f\, \subset\, Y$ be the branch locus of $f$, i.e., the
finite subset of $Y$ over which $f$ fails to be \'etale. So the restriction
$$
f\big\vert_{f^{-1}(Y\setminus B_f)}\, :\, f^{-1}(Y\setminus B_f) \, \longrightarrow\, Y\setminus B_f
$$
is \'etale. The singular locus of $X$ (respectively, $Y$) will be denoted by $S_X$ (respectively,
$S_Y$). The map $f$ will be called \textit{\'etale over singular locus} (EOSL for short) if
\begin{equation}\label{e1}
B_f\cap (S_Y\cup f(S_X))\,\,=\,\, \emptyset.
\end{equation}
Therefore, $f$ is EOSL if $f$ is \'etale at every point of $S_X \cup f^{-1}(S_Y)$. Moreover, when $f$ is EOSL,
then $S_X\,=\,f^{-1}(S_Y)$.

\begin{lemma}\label{lem1}
Let $f\, :\, X\, \longrightarrow\, Y$ be an EOSL map. Then the following two hold:
\begin{enumerate}
\item The direct image $f_*{\mathcal O}_X$ is locally free on $Y$.

\item For any torsionfree sheaf $E$ on $Y$, the pullback $f^*E$ is torsionfree.
\end{enumerate}
\end{lemma}

\begin{proof} The map $f$ is flat and hence $f_*{\mathcal O}_X$ is locally free. To see this another way, note 
that $f_*{\mathcal O}_X$ is torsionfree and hence it is locally free on $Y\setminus S_Y$, where $S_Y$ is the 
singular locus of $Y$. Now since $f$ is \'etale over neighborhoods of points of $S_Y$ it follows immediately 
that $f_*{\mathcal O}_X$ is locally free on a neighborhood of $S_Y$. Hence $f_*{\mathcal O}_X$ is locally free 
on entire $Y$.

The second statement is evident.
\end{proof}

Take an EOSL map $f\, :\, X\, \longrightarrow\, Y$. From Lemma \ref{lem1} we know that
$f_*{\mathcal O}_X$ is locally free. Let
\begin{equation}\label{e2}
F_1 \, \subset\, F_2\, \subset\, \cdots \, \subset\, F_m\,=\, f_*{\mathcal O}_X
\end{equation}
be the Harder--Narasimhan filtration of $f_*{\mathcal O}_X$ (see \cite{HL}). Note that $m\,=\,1$ if
$f_*{\mathcal O}_X$ is semistable. The subsheaf $F_1\, \subset\, f_*{\mathcal O}_X$ in \eqref{e2} is
called the maximal semistable subsheaf of $f_*{\mathcal O}_X$, and $\frac{\text{degree}(F_1)}{\text{rank}(F_1)}\, \in\, {\mathbb Q}$
is denoted by $\mu_{\rm max}(f_*{\mathcal O}_X)$ \cite{HL}. In general, $\frac{\text{degree}(V)}{\text{rank}(V)}\, \in\, {\mathbb Q}$
is denoted by $\mu (V)$.

Since $f_*{\mathcal O}_X$ is locally free, the pullback $f^*f_*{\mathcal O}_X$ is locally free.
In view of this, the proof of the following lemma is identical to the
proof in the special case where both $X$ and $Y$ are smooth \cite{BP}.

\begin{lemma}\label{lem2}
For the subsheaf $F_1$ in \eqref{e2},
$${\rm degree}(F_1)\,=\, 0.$$
\end{lemma}

\begin{proof}
This is proved in \cite[p.~ 12825, (2.7)]{BP} assuming that $X$ and $Y$ are smooth. Since this lemma turns out to be
crucial here, we give the details of its proof.

Since $f^*{\mathcal O}_Y\,=\, {\mathcal O}_X$, it follows from the adjunction formula (see \cite[p.~110]{Ha}) that
$$
{\mathcal O}_Y\,\, \subset\,\, f_*{\mathcal O}_X.
$$
This implies that
\begin{equation}\label{n1}
\mu(F_1)\,=\, \mu_{\rm max}(f_*{\mathcal O}_X) \, \geq\, \mu_{\rm max}({\mathcal O}_Y)\,=\, 0.
\end{equation}

On the other hand, a result on general vector bundles on $X$ says the following:
Since $f^*V$ is locally free for any torsionfree sheaf $V$ on $X$, for any vector bundle $\mathcal{V}$ on $X$,
$$
\mu_{\rm max}(f_*{\mathcal V})\,\, \leq\,\, \frac{\mu_{\rm max}(\mathcal{V})}{{\rm degree}(f)}
$$
(see \cite[p.~12824, Lemma 2.2]{BP}). Substituting ${\mathcal O}_X$ in place of $\mathcal V$ we conclude that
$$
\mu_{\rm max}(f_*{\mathcal V})\,\, \leq\,\, 0.
$$
This and \eqref{n1} together completes the proof.
\end{proof}

\begin{proposition}\label{prop1}
The subsheaf $F_1\, \subset\, f_*{\mathcal O}_X$ in \eqref{e2} is a subbundle, or in other words,
the quotient $(f_*{\mathcal O}_X)/F_1$ is locally free.
\end{proposition}

\begin{proof}
Let $\beta\, :\, \widehat{Y}\, \longrightarrow\, Y$ be the normalization of $Y$; so $\widehat{Y}$
is an irreducible smooth projective curve. Consider the fiber product
\begin{equation}\label{fp}
\begin{matrix}
\widehat{X}\,:=\, X\times_Y \widehat{Y} & \stackrel{\beta'}{\longrightarrow} & X\\
\,\,\, \Big\downarrow f' && \,\,\, \Big\downarrow f\\
\widehat{Y} & \stackrel{\beta}{\longrightarrow} & Y
\end{matrix}
\end{equation}
The given condition that $f$ is EOSL implies that $\widehat{X}$ in \eqref{fp} is smooth. Indeed, since $f$
is \'etale at $S_X\cup f^{-1}(S_Y)$, we have $f(S_X)\,=\,S_Y$. By base change $f'$ is \'etale over
$\beta^{-1}(S_Y)$. Hence all points in $(\beta')^{-1}(S_X)\,=\,(f')^{-1}(\beta^{-1}(S_Y))$ are smooth points
of $\widehat X$.

Since $f_*{\mathcal O}_X$ is locally free, the pullback
$\beta^*f_*{\mathcal O}_X$ is also locally free. We have
\begin{equation}\label{e3}
\beta^*f_*{\mathcal O}_X\,=\, f'_* {\mathcal O}_{\widehat{X}},
\end{equation}
where $f'$ is the map in \eqref{fp}. Let
$\widehat{F}\, \subset\, f'_* {\mathcal O}_{\widehat{X}}$ be the maximal semistable subsheaf (so
it is the first nonzero
term of the Harder--Narasimhan filtration of $f'_* {\mathcal O}_{\widehat{X}}$). We know that
${\rm degree}(\widehat{F})\,=\, 0$ \cite[p.~ 12825, (2.7)]{BP}. Therefore, from Lemma \ref{lem2} it
follows that for the isomorphism in \eqref{e3},
\begin{equation}\label{e4}
\beta^*f_*{\mathcal O}_X\, \supset \, \beta^* F_1\, \, \subset\, \widehat{F}
\, \subset\, f'_* {\mathcal O}_{\widehat{X}}.
\end{equation}

Note that the algebra structure of the sheaf ${\mathcal O}_{\widehat{X}}$ makes $f'_* {\mathcal O}_{\widehat{X}}$,
where $f'$ is as in \eqref{fp}, a sheaf of algebras over $\widehat Y$, and the corresponding spectrum is the (ramified) covering
$f'$. The subsheaf $\widehat{F}\, \subset\, f'_* {\mathcal O}_{\widehat{X}}$ in \eqref{e4} turns out to be a sheaf of subalgebras
(see \cite[p.~12826, Lemma 2.4]{BP}). Let $$\phi'\,\,:\, \,\widehat{Y}'\, \,\longrightarrow\,\, \widehat{Y}$$
be the (possibly ramified) covering map given by the spectrum of the sheaf of algebras $\widehat{F}$. Therefore,
we have
\begin{equation}\label{e9}
\widehat{F}\,\,=\, \,\phi'_*{\mathcal O}_{{\widehat{Y}'}}.
\end{equation}
The fact that $\widehat{F}$ is a sheaf of subalgebras of $f'_* {\mathcal O}_{\widehat{X}}$ implies that we have a morphism
\begin{equation}\label{e6}
q\, :\, \widehat{X}\, \longrightarrow\, \widehat{Y}'
\end{equation}
such that $\phi'\circ q\,=\, f'$ (see \cite[p.~ 12828, (2.11)]{BP} and the line following it). Note that this implies that
$\phi'_*{\mathcal O}_{{\widehat{Y}'}}\, \subset\, f'_* {\mathcal O}_{\widehat{X}}$, and \eqref{e9} implies that the
two subsheaves $\phi'_*{\mathcal O}_{{\widehat{Y}'}}$ and $\widehat{F}$ of $f'_* {\mathcal O}_{\widehat{X}}$ coincide.
Since $\text{degree}(\widehat{F})\,=\, 0$, using
\cite[p.~12825, Lemma 2.3]{BP} it follows that $\phi'$ is actually \'etale (see the lines following \cite[p.~12829, (2.12)]{BP}).

The above \'etale covering $\phi'$ of $\widehat{Y}$ produces an \'etale covering
\begin{equation}\label{e5}
\phi\, :\, Y'\, \longrightarrow\, Y.
\end{equation}
To see this, first note that the restriction of $\phi'$ to the complement
$\widehat{Y}'\setminus (\beta \circ\phi')^{-1}(S_Y)$ produces an \'etale covering
\begin{equation}\label{e7}
\phi_0\,:\, Y'_0\, \longrightarrow\, Y\setminus S_Y
\end{equation}
because the restriction
$$
\beta\big\vert_{\widehat{Y}\setminus \beta^{-1}(S_Y)}\, :\, \widehat{Y}\setminus \beta^{-1}(S_Y)
\, \longrightarrow\, Y\setminus S_Y
$$
is an isomorphism. The map $q$ in \eqref{e6} produces a map
$$
q_0\, :\, X\setminus f^{-1}(S_Y) \, \longrightarrow\, Y'_0
$$
(see \eqref{e7}). Indeed, $q_0$ is simply the restriction of $q$ to $\widehat{X}\setminus
(\beta\circ f')^{-1}(S_Y)$ (note that $X\setminus f^{-1}(S_Y)\,=\,\widehat{X}\setminus
(\beta\circ f')^{-1}(S_Y)$). Since $\phi'\circ q\,=\, f'$, it follows that
\begin{equation}\label{e8}
\phi_0\circ q_0 \,=\, f\big\vert_{X\setminus f^{-1}(S_Y)},
\end{equation}
where $\phi_0$ is the map in \eqref{e7}. Now from \eqref{e1} and \eqref{e8} it follows that
$\phi_0$ extends to an \'etale covering $\phi$ as in \eqref{e5}.

The identification of $\widehat{Y}'\setminus (\beta \circ\phi')^{-1}(S_Y)$ with $Y'_0$ extends to a map
$$
\beta_1\, :\, \widehat{Y}' \, \longrightarrow\, Y'
$$
because $\widehat{Y}'$ is smooth. Since the diagram
$$
\begin{matrix}
\widehat{Y}' & \stackrel{\beta_1}{\longrightarrow} & Y'\\
\,\,\, \Big\downarrow\phi' && \,\,\,\Big\downarrow\phi\\
\widehat{Y} & \stackrel{\beta}{\longrightarrow} & Y
\end{matrix}
$$
is Cartesian, we conclude that
\begin{equation}\label{e10}
\phi'_*{\mathcal O}_{\widehat{Y}'}\,=\, \beta^* (\phi_*{\mathcal O}_{Y'}).
\end{equation}

We have $\text{degree}(\phi_*{\mathcal O}_{Y'})\,=\, 0$ because $\phi$ is \'etale
(see \cite[p.\,13825, Lemma 2.3]{BP}). Hence from Lemma \ref{lem2} it follows that
$$
\phi_*{\mathcal O}_{Y'}\, \subset\, F_1.
$$
Consequently, \eqref{e10} implies that $\phi'_*{\mathcal O}_{\widehat{Y}'}\,\subset\, 
\beta^*F_1
$.
This and \eqref{e9} together give that $\widehat{F}\,\subset\, \beta^*F_1$. From this and
\eqref{e4} we conclude that
\begin{equation}\label{e11}
\beta^* F_1\, \, =\, \widehat{F}
\end{equation}
as subsheaves of $\beta^*f_*{\mathcal O}_X\,=\, f'_* {\mathcal O}_{\widehat{X}}$ (see \eqref{e3}).
On the other hand, from \eqref{e9} and \eqref{e10} we have $\widehat{F}\,=\, \beta^* (\phi_*{\mathcal O}_{Y'})$.
Combining this with \eqref{e11} it is deduced that $F_1\,=\, \phi_*{\mathcal O}_{Y'}$.
Now observe that $\phi_*{\mathcal O}_{Y'}$ is subbundle of $f_*{\mathcal O}_X$ because $\phi$ is \'etale and
$f$ is \'etale over a neighborhood of $S_Y\cup f(S_X)$. This completes the proof.
\end{proof}

\begin{corollary}\label{cor1}
The notation of the proof of Proposition \ref{prop1} is used.
\begin{enumerate}
\item ${\rm rank}(F_1)\,=\, {\rm rank}(\widehat{F})$.

\item $F_1\,=\, {\mathcal O}_Y$ if and only if $\widehat{F}\,=\, {\mathcal O}_{\widehat{Y}}$.
\end{enumerate}
\end{corollary}

\begin{proof}
The first statement follows immediately from \eqref{e11}.

Since ${\mathcal O}_Y\, \subset\, f_*{\mathcal O}_X$ and ${\mathcal O}_{\widehat{Y}}\, \subset\,
f'_*{\mathcal O}_{\widehat{X}}$, from Lemma \ref{lem2} it follows that
${\mathcal O}_Y\, \subset\, F_1$ and ${\mathcal O}_{\widehat{Y}}\, \subset\, \widehat{F}$. Therefore,
$F_1\,=\, {\mathcal O}_Y$ (respectively, $\widehat{F}\,=\, {\mathcal O}_{\widehat{Y}}$) if and
only if ${\rm rank}(F_1)\,=\, 1$ (respectively, ${\rm rank}(\widehat{F})\,=\, 1$). Now the second statement
follows from the first statement.
\end{proof}

\begin{proposition}\label{prop2}
Let $f\, :\, X\, \longrightarrow\, Y$ be an EOSL map. For any semistable vector bundle $E$ on $Y$
the pullback $f^*E$ is also semistable.
\end{proposition}

\begin{proof}
As before, $B_f\, \subset\, Y$ is the finite subset over which $f$ fails to be \'etale. Let $Y^o\,=\,Y\setminus B_f$.
Consider the \'etale cover $$f'\, :=\, f\big\vert_{X^o}\, :\,
X^o\, \longrightarrow\, Y^o$$ obtained by restricting $f$
to the complement $X^o:=X\setminus f^{-1}(B_f)$. Let
$$
f''\, :\, Z^o\, \longrightarrow\, Y^o
$$
be the Galois closure of $f'$ with the Galois group ${\rm Gal}(f'')$ being denoted by $G$. Since the map $f$ is EOSL,
it can be shown that the above map $f''$ extends to a ramified $G$--Galois cover
$$
\widetilde{f}\,\, :\,\, Z\,\, \longrightarrow\,\, Y,
$$
where $Z$ is a projective curve containing $Z^o$. To prove this, let
$y_1,\,\cdots,\, y_r$ be the singular points of $Y$. Let $\widehat{Y}_i$ be the formal neighborhood of 
$y_i$ in $Y$ and $\widehat{Y}_i^o\,=\,Y^o\times_Y Y_i$ for $1\,\le\, i\,\le \,r$. Note that for each $i\,\in\, \{1,\,
\cdots,\, r\}$, the map $f$ is \'etale 
over $y_i$ and $f''$ is the Galois closure of $f'$ which is the restriction of $f$. Hence the pullback of $f''$ along 
$\widehat{Y}_i^o \,\longrightarrow\, Y^o$ gives an isomorphism of $G$--Galois covers $$Z^o\times_{Y^o} \widehat{Y}_i^o 
\,\longrightarrow\, \widehat{Y}_i^o$$ with the $G$--Galois cover $\Ind^G_{\{e\}} \widehat{Y}_i^o \,\longrightarrow\, 
\widehat{Y}_i^o$ induced from the trivial cover defined by the identity map on $\widehat{Y}_i^o$. These isomorphisms allow us to 
patch $G$-Galois covers
$$\bigcup_{i=1}^r{\rm Ind}^G_{\{e\}} \widehat{Y}_i \,\,\,\longrightarrow\,\,\, \bigcup_{i=1}^r \widehat{Y}_i$$ and 
$Z^o \,\longrightarrow\, Y^o$ along $\bigcup_{i=1}^r\Ind^G_{\{e\}} \widehat{Y}_i^o \,\longrightarrow\, 
\bigcup_{i=1}^r 
\widehat{Y}_i^o$ to obtain the $G$-cover $\widetilde{f}\,:\, Z \,\longrightarrow\, Y$ (using \cite[Theorem 
3.1.9]{harbater} with the category of modules replaced by category of $G$-covers as in \cite[Theorem 
3.2.4]{harbater}).

Note that $\widetilde{f}$ is \'etale on a neighborhood
of $S_Y$. So $\widetilde f$ is flat. Let
$$
\phi\,:\,Z\, \longrightarrow\, X
$$ 
be the map such that $\widetilde{f}\,=\, f\circ\phi$.

Take any semistable vector bundle $E$ on $Y$. Assume that $f^*E$ is not semistable. Let $V\, \subset\,
f^*E$ be a subsheaf that destabilizes $f^*E$. Then
$\phi^*V$ destabilizes $\phi^*f^*E\,=\, \widetilde{f}^*E$.

Let $W\, \subset\, \widetilde{f}^*E$ be the
maximal semistable subsheaf of $\widetilde{f}^*E$ (in other words, it is the first nonzero term in the
Harder--Narasimhan filtration of $\widetilde{f}^*E$). Note that the natural action of the Galois
group $\text{Gal}(\widetilde{f})$ on $\widetilde{f}^*E$ preserves the above subsheaf $W$. Indeed, this
follows immediately from the uniqueness of the Harder--Narasimhan filtration.

The restriction of $W$ to
$Z\setminus {\widetilde{f}}^{-1}(S_Y)$ descends $Y\setminus S_Y$. On the other hand, the map
$f$ is \'etale over a neighborhood $U$ of $S_Y$, so the restriction of $W$ to $f^{-1}(U)$ descends to $U$.
Consequently, $W$ descends to a subsheaf of $E$. Since $W$ destabilizes $\widetilde{f}^*E$, it follows
immediately that this descend of $W$ to a subsheaf of $E$ destabilizes $E$. But $E$ is semistable.
In view of this contradiction we conclude that $f^*E$ is semistable.
\end{proof}

\section{Genuinely ramified maps}

Let $f\, :\, X\, \longrightarrow\, Y$ be an EOSL map. Consider the maximal semistable subsheaf $F_1\,\subset\,
f_*{\mathcal O}_X$. From Proposition \ref{prop1} we know that $F_1$ is a subbundle of $f_*{\mathcal O}_X$.

\begin{definition}\label{def1}
An EOSL map 
$f\, :\, X\, \longrightarrow\, Y$ will be called
\textit{genuinely ramified} if $F_1\,=\, {\mathcal O}_Y$.
\end{definition}

{}From Corollary \ref{cor1}(2) we know that $f$ is genuinely ramified if and only if the map
$f'\, :\, \widehat{X}\, \longrightarrow\, \widehat{Y}$ in \eqref{fp} is genuinely ramified. From
the proof the Proposition \ref{prop1} it follows that
the homomorphism of \'etale fundamental groups $$f_*\, :\, \pi_1^{\rm et}(X)
\,\longrightarrow\,\pi_1^{\rm et}(Y)$$ induced by $f$ is surjective if and only if
the homomorphism of \'etale fundamental groups $$f'_*\, :\, \pi_1^{\rm et}(\widehat{X})
\,\longrightarrow\,\pi_1^{\rm et}(\widehat{Y})$$ induced by $f'$ in \eqref{fp} is surjective.

\begin{theorem}\label{thm1}
Let $f\, :\, X\, \longrightarrow\, Y$ be an EOSL map between projective curves.
Then the following six statements are equivalent:
\begin{enumerate}
\item $f$ is genuinely ramified.

\item The map
$f'\, :\, \widehat{X}\, \longrightarrow\, \widehat{Y}$ in \eqref{fp} is genuinely ramified.

\item The homomorphism of \'etale fundamental groups $$f_*\, :\, \pi_1^{\rm et}(X)
\,\longrightarrow\,\pi_1^{\rm et}(Y)$$ induced by $f$ is surjective.

\item There is no nontrivial \'etale covering $\phi\, :\, Y'\, \longrightarrow\, Y$ admitting
a morphism $q\,:\, X\, \longrightarrow\, Y'$ such that $\phi\circ q \,=\, f$.

\item The fiber product $X\times_Y X$ is connected.

\item $\dim H^0(X,\, f^*f_* {\mathcal O}_X)\,=\, 1$.
\end{enumerate}
\end{theorem}

\begin{proof}
It was shown that the first two statements are equivalent. The third and fourth statements are clearly equivalent.

To show that the fifth and sixth statements are equivalent, consider the fiber product
$$
\begin{matrix}
X\times_Y X & \stackrel{\varphi}{\longrightarrow} & X\\
\,\,\,\, \Big\downarrow\beta && \,\,\, \Big\downarrow f\\
X & \stackrel{f}{\longrightarrow} & Y
\end{matrix}
$$
We have $\beta_* {\mathcal O}_{X\times_Y X}\,=\, f^*f_* {\mathcal O}_X$. Since the above diagram is Cartesian, we have
$$
f^*f_* {\mathcal O}_X\,=\, \beta_*\varphi^*{\mathcal O}_X\,=\, \beta_* {\mathcal O}_{X\times_Y X},
$$
and hence
\begin{equation}\label{f3}
H^0(X,\, f^*f_* {\mathcal O}_X)\,=\, H^0(X,\, \beta_* {\mathcal O}_{X\times_Y X})\,=\,
H^0(X\times_Y X,\, {\mathcal O}_{X\times_Y X}).
\end{equation}
Since $X\times_Y X$ is connected if and only if $\dim H^0(X\times_Y X,\, {\mathcal O}_{X\times_Y X})
\,=\, 1$, from \eqref{f3} it follows that the fifth and sixth statements are equivalent.

To show that the second and third statements are equivalent, recall that the third statement
holds if and only if the homomorphism of \'etale fundamental groups $$f'_*\, :\, \pi_1^{\rm et}(\widehat{X})
\,\longrightarrow\,\pi_1^{\rm et}(\widehat{Y})$$ induced by $f'$ in \eqref{fp} is surjective. But this
homomorphism $f'_*$ is surjective if and only if $f'$ is genuinely ramified \cite[p.~12828, Proposition 2.6]{BP}.
So the second and third statements are equivalent.

We will now show that the first statement implies the sixth statement. From Proposition \ref{prop2}
we conclude that for any vector bundle $V$ on $Y$, the Harder--Narasimhan filtration of $f^*V$ is simply
the pullback, by $f$, of the Harder--Narasimhan filtration of $V$.

Assume that $f$ is genuinely ramified.
This implies that the Harder--Narasimhan filtration of $f_*{\mathcal O}_X$ in \eqref{e2} is of the form
$$
{\mathcal O}_Y \,=\, F_1 \, \subset\, F_2\, \subset\, \cdots \, \subset\, F_m\,=\, f_*{\mathcal O}_X,
$$
where $\text{degree}(F_j/F_{j-1})\, <\, 0$ for all $2\, \leq\, j\, \leq\, m$. Consequently, the
Harder--Narasimhan filtration of $f^*f_*{\mathcal O}_X$ is the following:
\begin{equation}\label{f1}
{\mathcal O}_X \,=\, f^*F_1 \, \subset\, f^*F_2\, \subset\, \cdots \, \subset\, f^*F_m\,=\, f^*f_*{\mathcal O}_X,
\end{equation}
Since $\text{degree}(f^*F_j/f^*F_{j-1})\, =\,\text{degree}(f)\cdot \text{degree}(F_j/F_{j-1})\, <\, 0$,
we have $$H^0(X,\, (f^*F_j)/(f^*F_{j-1}))\,=\, 0$$ for all $2\, \leq\, j\, \leq\, m$. In view of this, from
\eqref{f1} it follows that $$H^0(X,\, f^*f_*{\mathcal O}_X)\,=\, H^0(X,\, f^*F_1)\,=\, H^0(X,\, {\mathcal O}_X).$$
Hence the sixth statement holds.

Finally, we will show that the sixth statement implies the fourth statement. To prove this by contradiction,
let $\phi\, :\, Y'\, \longrightarrow\, Y$ be a nontrivial \'etale covering, and
$q\,:\, X\, \longrightarrow\, Y'$ a morphism, such that $\phi\circ q \,=\, f$. Then we have
$$
\phi_*{\mathcal O}_{Y'} \,\, \subset\,\, f_*{\mathcal O}_X,
$$
and hence
$$
q^*\phi^* \phi_*{\mathcal O}_{Y'}\,=\, f^*\phi_*{\mathcal O}_{Y'}\,\subset\, f^*f_*{\mathcal O}_X.
$$
This implies that
\begin{equation}\label{f2}
\dim H^0(X,\, f^*f_* {\mathcal O}_X)\,\geq \,\dim H^0(Y',\, \phi^* \phi_*{\mathcal O}_{Y'}).
\end{equation}

Since $\phi$ is \'etale, $Y'$ is a connected component of $Y'\times_Y Y'$ using the diagonal map.
So $Y'\times_Y Y'$ is not connected. Hence setting $f\,=\, \phi$ in \eqref{f3} we conclude that
$$\dim H^0(Y',\, \phi^* \phi_*{\mathcal O}_{Y'})\, \geq\, 2.$$ Therefore, \eqref{f2} contradicts the
sixth statement. So the sixth statement implies the fourth statement. This completes the proof.
\end{proof}

\begin{lemma}\label{lem3}
Let $f\, :\, X\, \longrightarrow\, Y$ be an EOSL map of degree $d$ between projective curves.
Assume that there is a finite group $\Gamma$ acting 
faithfully on $X$ such that $Y\,=\, X/\Gamma$. If $f$ is genuinely ramified, then
$$
f^*((f_*{\mathcal O}_X)/{\mathcal O}_Y)\,=\, (f^*f_*{\mathcal O}_X)/{\mathcal O}_X\,\, \hookrightarrow\,\,
\bigoplus_{i=1}^{d-1} {\mathcal L}_i,
$$
where each ${\mathcal L}_i$ is a line bundle on $X$ of negative degree.
\end{lemma}

Lemma \ref{lem3} was proved in \cite{BP} under the assumption that $X$ is smooth (see \cite[p.~12837, 
Proposition 3.5]{BP}). The same proof works here.

\begin{theorem}\label{thm2}
Let $f\, :\, X\, \longrightarrow\, Y$ be an EOSL map of degree $d$ between
projective curves. Then the following two statements are equivalent:
\begin{enumerate}
\item $f$ is genuinely ramified;

\item $f^*E$ is stable for every stable vector sheaf $E$ on $Y$.
\end{enumerate}
\end{theorem}

Theorem \ref{thm2} is proved exactly as Theorem 5.3 of \cite[p.~12850]{BP} is proved.

\section{Some examples}\label{se-ex}

\subsection{Example 1}

Let $Y$ an irreducible nodal projective curve of arithmetic genus at least two. Let
$f\, :\, X\, \longrightarrow\, Y$ be the normalization. Then $f$ satisfies (4) and (5)
of Theorem \ref{thm1} but does not satisfy (1) and (3) of Theorem \ref{thm1}. Note that
$f$ is not an EOSL map.

\subsection{Example 2}

Consider the map
$$
\phi\, :\, {\mathbb C}{\mathbb P}^1\, \longrightarrow\, {\mathbb C}{\mathbb P}^1,\ \
\ z\, \longrightarrow\, z^2.
$$
Let $\psi_1\, :\, {\mathbb C}{\mathbb P}^1\, \longrightarrow\, X$ be the rational nodal
curve of arithmetic genus one obtained by identifying $1$ and $\sqrt{2}$.
Let $\psi_2\, :\, {\mathbb C}{\mathbb P}^1\, \longrightarrow\, Y$ be the rational nodal
curve of arithmetic genus one obtained by identifying $1$ and $2$. The map $\psi_2\circ\phi$
factors $\psi_1$. In other words, there is a unique map
$$f\, :\, X\, \longrightarrow\, Y$$
such that $\psi_2\circ\phi\,=\, f\circ\psi_1$. This map $f$ is clearly not EOSL.
Note that the homomorphism of \'etale fundamental groups
$$f_*\, :\, \pi_1^{\rm et}(X) \,\longrightarrow\,\pi_1^{\rm et}(Y)$$ induced by $f$ is surjective. So
statement (3) of Theorem \ref{thm1} holds.
We will show that there is a stable vector bundle on $Y$ whose pullback to $X$ is not stable.

Let $\beta\, :\, Z\ \longrightarrow\, Y$ be the unique \'etale covering of degree two (note that
$\pi_1(Y)\,=\, {\mathbb Z}$). Let $L$ be a holomorphic line bundle on $Z$ of degree one. Then
the direct image $\beta_*L$ is a vector bundle of rank two and degree one. To prove that
$\beta_*L$ is semistable, take any rank one subsheaf $F\, \subset\, \beta_*L$. Then we have a nonzero
homomorphism $\beta^*F \, \longrightarrow\, L$ because
$$
H^0(Y, \, {\rm Hom}(F,\,\beta_*L))\, \cong\, H^0(Z, \, {\rm Hom}(\beta^*F,\,L))
$$
(see \cite[p.~110]{Ha}). Since there is a nonzero homomorphism $\beta^*F \, \longrightarrow\, L$,
we conclude that
$$
2\cdot \text{degree}(F)\,=\, \text{degree}(\beta^*F)\, \leq\, \text{degree}(L)\,=\, 1.
$$
Therefore, it follows that $\beta_*L$ is semistable. Since $\text{degree}(\beta_*L)\,=\, 1$, this implies
that $\beta_*L$ is stable. So $f^*\beta_*L$ is a vector bundle on $X$ of rank two and degree two.

It can be shown that there is no stable vector bundle of rank two and degree two on $X$. Indeed, if $W$ is
a vector bundle on $X$ of rank two and degree two, then
$$
\dim H^0(X,\, W) \,\geq\, \dim H^0(X,\, W) - \dim H^1(X,\, W) \,=\, 2.
$$
Take two linearly independent sections $s$ and $t$ on $W$, and consider the evaluation homomorphism
$$
\eta\, :\, {\mathcal O}_X\oplus {\mathcal O}_X \, \longrightarrow\, W
$$
that sends any $(c_1,\, c_2)\, \in\, {\mathcal O}_x\oplus {\mathcal O}_x\, =\, {\mathbb C}^2$, $x\, \in\, X$,
to $c_1\cdot s(x)+ c_2\cdot t(x)\,\in\, W_x$. This $\eta$ is not an isomorphism
over $X$, because $\text{degree}(W)\,=\, 2\, >\, \text{degree}({\mathcal O}_X\oplus {\mathcal O}_X)$. Therefore,
there is $(a,\, b)\,\not=\, (0,\, 0)$ such that $as+bt$ vanishes at some point of $X$. The degree of the
rank one subsheaf of $W$ generated by $as+bt$ is at least one. Hence $W$ is not stable.
In particular, $f^*\beta_*L$ is not stable.

\subsection{Example 3}

Let
\begin{equation}\label{z1}
\gamma\,:\, {\mathbb C}{\mathbb P}^1\, \longrightarrow\, {\mathbb C}{\mathbb P}^1
\end{equation}
be the morphism defined by $z\, \longrightarrow\,
z^d$, with $d\, \geq\, 5$. Denote by $X$ the nodal curve of arithmetic genus $1$ obtained by identifying $1\, \in\, 
{\mathbb C}{\mathbb P}^1$ with $-1\, \in\, {\mathbb C}{\mathbb P}^1$. Let $Y$ be the nodal curve of arithmetic genus $1$
obtained by identifying $1\, \in\, {\mathbb C}{\mathbb P}^1$ with $\exp(\pi\sqrt{-1}/d)\, \in\, {\mathbb C}{\mathbb P}^1$.
The map $\gamma$ in \eqref{z1} produces a map
\begin{equation}\label{z2}
f\,:\, Y\, \longrightarrow\, X.
\end{equation}
Note that $\pi_1(Y,\, y_0)\,=\, {\mathbb Z}\,=\, \pi_1(X,\, f(y_0))$, and the induced homomorphism
\begin{equation}\label{e0}
f_*\,:\, \pi_1(Y,\, y_0)\,\longrightarrow\, \pi_1(X,\, f(y_0))
\end{equation}
is an isomorphism.

Consider the vector bundle ${\mathcal O}_{{\mathbb C}{\mathbb P}^1}\oplus {\mathcal O}_{{\mathbb C}{\mathbb P}^1}(1)$
on ${\mathbb C}{\mathbb P}^1$ of rank two and degree one. Take any isomorphism of fibers over $1$ and $-1$
\begin{equation}\label{z3}
I\, :=\, ({\mathcal O}_{{\mathbb C}{\mathbb P}^1}\oplus {\mathcal O}_{{\mathbb C}{\mathbb P}^1}(1))_1
\, \longrightarrow\,({\mathcal O}_{{\mathbb C}{\mathbb P}^1}\oplus {\mathcal O}_{{\mathbb C}{\mathbb P}^1}(1))_{-1}
\end{equation}
such that $I({\mathcal O}_{{\mathbb C}{\mathbb P}^1}(1)_1)\,=\,
({\mathcal O}_{{\mathbb C}{\mathbb P}^1})_{-1}$. Identifying the
fibers of ${\mathcal O}_{{\mathbb C}{\mathbb P}^1}\oplus {\mathcal O}_{{\mathbb C}{\mathbb P}^1}(1)$ over $-1$ and $-1$
using $I$ we obtain a vector bundle $E$ on $X$ of rank two and degree $1$. It is straightforward to check that $E$ is stable.

Now consider the vector bundle
$$
{\gamma}^*({\mathcal O}_{{\mathbb C}{\mathbb P}^1}\oplus{\mathcal O}_{{\mathbb C}{\mathbb P}^1}(1))\, \cong\,
{\mathcal O}_{{\mathbb C}{\mathbb P}^1}\oplus {\mathcal O}_{{\mathbb C}{\mathbb P}^1}(d)
$$
on ${\mathbb C}{\mathbb P}^1$, where $\gamma$ is the map in \eqref{z1}. Consider
the following isomorphism of its fibers over $1$ and $\exp(\pi\sqrt{-1}/d)$:
$$
{\gamma}^*({\mathcal O}_{{\mathbb C}{\mathbb P}^1}\oplus{\mathcal O}_{{\mathbb C}{\mathbb P}^1}(1))_1
\,=\, ({\mathcal O}_{{\mathbb C}{\mathbb P}^1}\oplus {\mathcal O}_{{\mathbb C}{\mathbb P}^1}(1))_1
\,\stackrel{I}{\longrightarrow}\,
({\mathcal O}_{{\mathbb C}{\mathbb P}^1}\oplus {\mathcal O}_{{\mathbb C}{\mathbb P}^1}(1))_{-1}
$$
$$
=\,\,
{\gamma}^*({\mathcal O}_{{\mathbb C}{\mathbb P}^1}\oplus {\mathcal O}_{{\mathbb C}{\mathbb P}^1}(1))_{\exp(\pi\sqrt{-1}/d)},
$$
where $I$ is the isomorphism in \eqref{z3}. Identifying the
fibers of ${\gamma}^*({\mathcal O}_{{\mathbb C}{\mathbb P}^1}\oplus
{\mathcal O}_{{\mathbb C}{\mathbb P}^1}(1))$ over $-1$ and $\exp(\pi\sqrt{-1}/d)$
using this isomorphism we obtain a vector bundle $V$ on $Y$ of rank two and degree $d$. Note that we have
\begin{equation}\label{z5}
f^*E\,=\, V,
\end{equation}
where $f$ is the map in \eqref{z2}.

We will construct a subsheaf of $f^*E$ of rank one and degree $d-2$. Consider the subsheaf
$$
{\mathcal O}_{{\mathbb C}{\mathbb P}^1}(d-2)\, \cong\, ({\gamma}^* ({\mathcal O}_{{\mathbb C}{\mathbb P}^1}(1)))
\otimes{\mathcal O}_{{\mathbb C}{\mathbb P}^1}(-1- \exp(\pi\sqrt{-1}/d))
$$
$$
\subset\,\,{\gamma}^* ({\mathcal O}_{{\mathbb C}{\mathbb P}^1}(1))\, \subset\,
{\gamma}^*({\mathcal O}_{{\mathbb C}{\mathbb P}^1}\oplus
{\mathcal O}_{{\mathbb C}{\mathbb P}^1}(1)).
$$
It produces a subsheaf of $V$ of rank $1$ and degree $d-2$. Now using the isomorphism in \eqref{z5} this subsheaf
produces a subsheaf of $f^*E$ of rank $1$ and degree $d-2$. Consequently, the vector bundle $f^*E$ is not stable
(recall that $d\, \geq\, 5$), although $E$ is stable and the homomorphism in \eqref{e0} is an isomorphism.

\section*{Acknowledgements}

We thank the referee for a very careful reading.

\end{document}